\documentclass[a4paper,12pt]{amsart}
\usepackage[mathscr]{eucal}
\usepackage{amsmath, amsfonts, amssymb, amscd, amsthm}
\usepackage[english]{babel}
\usepackage{mathabx}
\usepackage{tikz-cd}
\usepackage{comment}

\usepackage[colorlinks]{hyperref}
\definecolor{linkblue}{RGB}{1,1,190}
\definecolor{citered}{RGB}{190,1,1}
\hypersetup{
	linkcolor=linkblue,
	urlcolor=linkblue,
	citecolor=citered
}

\theoremstyle{plain}
\newtheorem{theorem}{\bf Theorem}[section]
\newtheorem{proposition}[theorem]{\bf Proposition}
\newtheorem{lemma}[theorem]{\bf Lemma}

\newtheorem{conjecture}[theorem]{\bf Conjecture}

\theoremstyle{definition}

\newtheorem{definition}[theorem]{\bf Definition}

\newtheorem{question}[theorem]{\bf Question}

\theoremstyle{definition}

\theoremstyle{definition}
\newtheorem{claim}{}[theorem]
\newtheorem{clam}{}[theorem]

\newcommand{\N}{\mathbb N}
\newcommand{\Z}{\mathbb Z}

\newcommand{\Q}{\mathbb Q}

\newcommand{\C}{\mathbb C}

 \DeclareMathOperator{\ord}{ord}
\DeclareMathOperator{\lcm}{lcm} 
 
 \DeclareMathOperator{\supp}{supp}
  
\DeclareMathOperator{\Hom}{Hom}

\DeclareMathOperator{\sep}{sep}

\renewcommand{\t}{\, | \,}

\newcommand{\be}{\begin{equation}}
\newcommand{\ee}{\end{equation}}

\newcommand{\ber}{\begin{eqnarray}}
\newcommand{\eer}{\end{eqnarray}}

\newcommand{\sepbeta}{\beta_{\sep}}

\numberwithin{equation}{section}

\begin{document}
\title[On separating sets]{On separating sets of polynomial invariants of finite abelian group actions}

\author{Barna Schefler and Kevin Zhao$^\ast$ and Qinghai Zhong}

\address{School of Mathematics and Statistics, Nanning Normal University, Nanning 530100, China, and Center for Applied Mathematics of  Guangxi, Nanning Normal University, Nanning 530100, China}
\email{zhkw-hebei@163.com}

\address{E\"otv\"os Lor\'and University,
P\'azm\'any P\'eter s\'et\'any 1/C, 1117 Budapest, Hungary}
\email{scheflerbarna@yahoo.com}

\address{University of Graz, NAWI Graz\\
	Department of Mathematics and Scientific Computing\\
	Heinrichstra{\ss}e 36\\
	8010 Graz, Austria}
\email{qinghai.zhong@uni-graz.at}
\urladdr{https://imsc.uni-graz.at/zhong/}

\thanks{$\ast$ Corresponding author}
\thanks{This research was funded in part by National Science Foundation
	of China Grant \#12301425, by the Austrian Science Fund (FWF) [grant DOI:10.55776/P36852] and by the Hungarian National Research, Development and Innovation Office,  NKFIH K 138828.
}


\begin{abstract}
Let $G$ be a finite group acting on a finite dimensional complex vector space $V$ via linear transformations. Let $\mathbb{C}[V]^G$ be the algebra of polynomials that are invariant under the induced $G$-action on the polynomial ring $\mathbb{C}[V]$. A subset $S\subseteq\mathbb{C}[V]^G$ is a separating set if it separates the orbits of the group action. If $G$ is abelian, then there exist finite separating sets consisting of monomials. In this paper we investigate properties of separating sets from three different points of view, including
the minimal size of separating sets consisting of monomials, the exact value of the separating Noether number $\sepbeta(G)$ of abelian groups of rank $4$, and an inverse problem of $\sepbeta(G)$ for abelian groups of rank $2$.
\end{abstract}

\maketitle

\section{Introduction}
\subsection{Invariants for finite abelian group actions}
Let $V$ be an $n$-dimensional vector space over the complex field $\mathbb{C}$ and  let $G$ be a finite group. Suppose that $V$ is endowed with an action of $G$ via linear transformations (i.e. $V$ has a $\mathbb{C}G$-module structure). The $G$-action on $V$ induces a $G$-action on the coordinate ring $\mathbb{C}[V]$ of $V$:
\begin{center}
	for $g\in G$ and $f\in\mathbb{C}[V]$, we have: $g\cdot f(v)=f(g^{-1}\cdot v)$.
\end{center}
By the famous theorem of Noether \cite{N16}, the invariant algebra
\begin{equation*}
\mathbb{C}[V]^G=\{f\in\mathbb{C}[V]\colon f=g\cdot f \mbox{ for any }g\in G\}
\end{equation*}
is finitely generated by homogeneous polynomials that have degree at most $|G|$.

Studying properties of separating invariants became a popular topic within invariant theory in recent years \cite{Do17a, domokos-schefler, draisma-kemper-wehlau, dufresne, Ko-Kr10a, SZZ25a}. Recall that a subset $S\subseteq \mathbb{C}[V]^G$ is called a \emph{separating set} if the following holds:
\begin{center}
	 for any two distinct $v_1, v_2\in V$, there exists $f\in S$ such that $f(v_1)\neq f(v_2)$,\\ whenever
	 there is $h\in\mathbb{C}[V]^G$ such that $h(v_1)\neq h(v_2)$.
\end{center}
If $G$ is a \textit{finite} group, then  $S\subseteq \mathbb{C}[V]^G$ is a separating set if and only if it separates the orbits of the group action, in other words,
\begin{center}
	$Gv_1\neq Gv_2$ implies the existence of an $f\in S$ such that $f(v_1)\neq f(v_2)$.
\end{center}
Since generating sets are separating sets, Noether's theorem implies  that a finite separating set also exists. Therefore, it is natural to ask the following questions:
\begin{question}\label{q2}
    What is a sharp lower bound for the size of a separating set?
\end{question}
\begin{question}\label{q1}
    What is a sharp upper bound for the degrees of the polynomials appearing in a separating set?
\end{question}

One concept of main interest of this paper is the \emph{separating Noether number} $\sepbeta(G)$. It was introduced in \cite{Ko-Kr10a} in order to deal with Question \ref{q1}.
\begin{definition}\label{sepnoetherdef}
	Let $G$ be a finite group and let $V$ be a finite dimensional $\mathbb{C}G$-module.
Denote by ${\sepbeta}(G,V)$ the minimal positive integer $d$ such that $\mathbb{C}[V]^G$ contains a separating set consisting of homogeneous polynomials of degree at most $d$. The separating Noether number of the group $G$ is
\[
\sepbeta(G):=\sup\{{\sepbeta}(G,V) \colon V \text{ is a finite dimensional }\mathbb{C}G\text{-module}\}. \]
\end{definition}
Definition \ref{sepnoetherdef} was inspired by the definition of the \emph{Noether number} $\beta(G)$ introduced in \cite{Sc91a} as follows:
\begin{align*}
    \beta(G,V):&=\min\{d\in \mathbb{N}_0\colon \mathbb{C}[V]^G \text{ is generated by  polynomials of degree }\leq d\},\\
 \text{ and }\quad   \beta(G):&=\sup\{\beta(G,V)\colon V\text{ is a finite dimensional }\mathbb{C}G\text{-module}\}.
\end{align*}

Let $G$ be a finite abelian group and let $\widehat{G}_{add}=\Hom (G, \Q/\Z)$. Then $\widehat{G}_{add}\cong G$ is a finite abelian group and the action of $G$ on $V$ is diagonalizable, so there exists a basis $(x_1, \ldots, x_n)$ of the dual space $V^*$ of $V$  such that for every $i\in [1,n]$, we have  $g\cdot x_i=e^{2\pi i\chi_i(g)}x_i$ for some $\chi_i\in \widehat{G}_{add}$. Denote $\widehat{G}_V=\{\chi_1,\ldots, \chi_n\}$. Clearly, the assignment $V\mapsto \widehat{G}_V$ induces a bijection between the set of isomorphism classes of multiplicity-free $\mathbb{C}G$-modules and the power set $\mathscr{P}(\widehat{G}_{add})$ of $\widehat{G}_{add}$. Under this bijection, the regular representation is mapped to $\widehat{G}_{add}$.

For every finite abelian group $(H,+,0)$,  we denote by $\mathcal{F}(H)$ the free abelian monoid with basis $H$ and by $\mathcal B(H)$ the set of all zero-sum sequences over $H$. Write $\psi$ for the map from the set $\mathcal{M}[V]$ of monomials into $\mathcal{F}(\widehat{G}_{add})$ assigning to the monomial $x_1^{m_1}\ldots x_n^{m_n}$ the sequence containing $\chi_i$ with multiplicity $m_i$, that is $\psi(x_1^{m_1}\cdot\ldots\cdot x_n^{m_n})=\chi_1^{m_1}\cdot\ldots\cdot \chi_n^{m_n}\in \mathcal{F}(\widehat{G}_{add})$. It is well known that a monomial $x_1^{m_1}\cdot\ldots\cdot x_n^{m_n}$ is $G$-invariant if and only $\psi(x_1^{m_1}\cdot\ldots\cdot x_n^{m_n})\in\mathcal{B}(\widehat{G}_{add})$, and $\mathbb{C}[V]^G$ is spanned by the set $\mathcal{M}[V]^G$ of $G$-invariant monomials. Moreover, if $V$ is multiplicity-free, then $\psi$ is injective, so it gives an isomorphism between the monoids $\mathcal{M}[V]^G$ and $\mathcal{B}(\widehat{G}_{add})$, when $V$ is the regular $\C G$-module. As a consequence, $\beta(G) = \mathsf{D}(\widehat{G}_{add}) = \mathsf{D}(G)$, where the \textit{Davenport constant} $\mathsf{D}(G)$ is the maximal length of a minimal zero-sum sequence over $G$. One can check \cite{Cz-Do-Ge16} for more details.

\subsection{Main results}

The main aim of this paper is to analyze separating sets consisting of monomials, which we call \emph{monomial separating sets},  through Questions \ref{q2} and \ref{q1}.

In recent years, the separating Noether number has been studied a lot \cite{Do17a, domokos-schefler, domokos-schefler2, Sc24a, Sc25a, SZZ25a}.  Among others, the exact value of $\sepbeta(G)$ was given for groups of small order \cite{domokos-schefler,domokos-schefler2}, for non-commutative groups containing cyclic subgroups of index $2$ \cite{domokos-schefler2}, for direct sum of several copies of the cyclic group $C_n$ \cite{Sc24a}, and for $p$-groups \cite{SZZ25a}. Moreover, in \cite{SZZ25a} the authors gave the value of $\sepbeta(G)$ for finite abelian groups $G$ that have rank $2,3$ or $5$. Our result fills this missing gap between $3$ and $5$:
\begin{theorem}\label{thm:rank4}
	Let $G=C_{n_1}\oplus C_{n_2}\oplus C_{n_3}\oplus C_{n_4}$ be a finite abelian group of rank $4$ with $1<n_1\mid n_2\mid n_3\mid n_4$. Then
	$${\sepbeta}(G)=\frac{n_{2}}{p}+n_{3}+n_4,$$
	where $p$ is the minimal prime divisor of  $n_1$.
\end{theorem}

Inverse problems of zero-sum theory are classical topics in additive combinatorics \cite{bhs-p, GG03, GGG10, GS20, W10, S11}. The problem of this type asks for the structure of zero-sum sequences having some specific properties. For example, the inverse problem of $\beta(G)=\mathsf D(G)$ is to characterize the structure of minimal zero-sum sequences of length $\mathsf D(G)$ over $G$. Note that the inverse problem of $\beta(G)$ for rank $2$ groups was a giant task, that was solved in a series of five articles \cite{bhs-p, GG03, GGG10, Re07, W10}.
 In fact, the separating Noether number $\beta_{sep}(G)$ is the maximal length of a separating atom (which is a special minimal zero-sum sequence, see Definition \ref{groupatomdef}) over some subset $G_0\subseteq G$ (see Lemma \ref{md2}).
In Section \ref{sec:invprob}, we obtain an inverse result for $\beta_{sep}(G)$ that is a step towards the full characterization of separating atoms of maximal length over a rank $2$ abelian group. More precisely, we prove that:
\begin{theorem}\label{max2}
	Let $G=C_{n_1}\oplus C_{n_2}$ with $1<n_1\t n_2$ and
	let $A$ be a separating atom  with $|A|={\sepbeta}(G)$ and $|\supp(A)|\le 3$. Then there exist three elements $g_1,g_2,g_3\in G$ with $\ord(g_1)=\ord(g_2)=n_2$ and $\ord(g_3)=n_1$ such that  $\supp(A)=\{g_1,g_2,g_3\}$ and
 $g_i\not \in \langle g_j\rangle$ for any two distinct indexes $i,j\in [1,3]$.
\end{theorem}

\medskip
 \emph{For the remainder of the manuscript, let $G$ be a finite abelian group written additively, let $\widehat{G}_{add}=\Hom (G, \Q/\Z)$ be the dual group of $G$ written additively,  let $V$ be the regular $\C G$-module, and let $\psi$ be the isomorphism between the monoids $\mathcal M[V]^G$ and $\mathcal B(\widehat{G}_{add})$.}

\section{Preliminaries}\label{sec:separating-Davenport}

Let $\mathbb N$ denote the set of positive integers and $\mathbb{N}_0=\mathbb{N}\cup\{0\}$. For two real numbers $a, b \in \mathbb R$, we denote by  $[a, b] = \{ x \in
\mathbb Z \colon a \le x \le b\}$ the finite discrete interval. For $n \in \N$ we denote by $C_n$ the cyclic group of order $n$. Since $G$ is abelian, we have $G \cong C_{n_1} \oplus \ldots \oplus C_{n_r}$, where $r \in \N_0$ and $n_1, \ldots, n_r \in \N$ with $1 < n_1 \t \ldots \t n_r$. We call $r = \mathsf r (G)$ the {\it rank} of $G$, $n_r=\exp(G)$ the {\it exponent} of $G$, and a tuple $(e_1, \ldots, e_r)$ of nonzero elements of $G$ is said to be a {\it basis} if $G = \langle e_1 \rangle \oplus \ldots \oplus \langle e_r \rangle$ with $\langle e_i \rangle\cong C_{n_i}$.

By a   {\it monoid},  we  mean a commutative
semigroup with identity which satisfies the cancellation law (that
is, if $a,b ,c$ are elements of the monoid with $ab = ac$, then $b =
c$ follows). The multiplicative semigroup of non-zero elements of an
integral domain is a monoid.
Let $H$ be a monoid. We
denote by $H^{\times}$ the group of invertible elements of $H$, by  $\mathcal A (H)$  the   set of atoms (irreducible elements) of $H$, and
by $\mathsf q(H)$ the quotient group of $H$. If $H^{\times}=\{1\}$, we say $H$ is \emph{reduced}.

\subsection{Zero-sum sequences}\label{subsec:zss}
Let $G_0\subseteq G$ be a nonempty subset. We denote by $\langle G_0\rangle$ the group generated by $G_0$. In additive combinatorics, a \emph{sequence} over $G_0$ means a finite unordered sequence
with terms from $G_0$, where repetition is allowed.
Let \[S=g_1\ldots g_{\ell}=\prod_{g\in G_0}g^{\mathsf v_g(S)}\] be a sequence over $G_0$, where $\mathsf v_g(S)$ denotes the multiplicity of $g$ in $S$. In other words, sequences are elements of the multiplicatively written free abelian monoid $\mathcal F(G_0)$ with basis $G_0$. Let $T$ be another sequence over $G_0$. If $\mathsf v_g(T)\le \mathsf v_g(S)$ for every $g\in G_0$, then we say $T$ is a \emph{subsequence} of $S$ and denote it by $T\t S$. We also denote  $T^{-1}S=\prod_{g\in G_0}g^{\mathsf v_g(S)-\mathsf v_g(T)}$ the remaining sequence.
We call
\begin{align} \label{length}
  |S|  &= \ell = \sum_{g \in G} \mathsf v_g (S) \in \mathbb N_0 \text{ the {\it length} of $S$},\\ \notag
     \supp(S)&=\{g\in G_0\colon \mathsf v_g(S)\neq 0\} \text{ the {\it support} of }S,\\ \notag
     \sigma (S) & = \sum_{i = 1}^{\ell} g_i=\sum_{g\in G_0}\mathsf v_g(S)g\in G \text{ the {\it sum} of }S,  \\ \notag
	\Sigma (S) &= \Big\{ \sum_{i \in I} g_i
\colon \emptyset \ne I \subseteq \{1,\ldots,\ell\} \Big\} \text{ the {\it set of
		subsequence sums} of $S$}.
\end{align}
A sequence $S$ is called a {\it zero-sum sequence} \ if \ $\sigma (S) = 0$, and {\it zero-sum free} \ if \ $0 \notin \Sigma (S)$. It is easy to see that the set of all zero-sum sequences over $G_0$ forms a submonoid
\begin{center}
    $\mathcal{B}(G_0) := \{S \in \mathcal F (G_0) \colon \sigma (S)=0\} \subseteq \mathcal F (G_0)$.
\end{center}
For any subset $B\subseteq \mathcal{B}(G_0)$, we denote by $\langle B\rangle$ the quotient subgroup generated by $B$.  Obviously, $\langle B\rangle$ is a subgroup of the quotient group $\mathsf q(\mathcal{B}(G_0))$.

A nontrivial zero-sum sequence is called  {\it a minimal zero-sum sequence } or {\it an atom}  if its every proper subsequence is zero-sum free. It is easy to see that the set of all minimal zero-sum sequences over $G_0$ are exactly the set of atoms $\mathcal A(G_0):=\mathcal A(\mathcal B(G_0))$.
 Note that  if $A\in\mathcal{A}(G_0)$,  then $g^{-1}A$ is zero-sum free for any $g\mid A$.

The {\it Davenport constant} $\mathsf D(G_0)$ of the monoid $\mathcal{B}(G_0)$ is the maximal length of atoms over $G_0$, that is,
\[
\mathsf D(G_0)=\max \{|A|\colon A\in \mathcal A(G_0)\}\,.
\]
So for every zero-sum free sequence $S$ over $G$, we have $|S|\le \mathsf D(G)-1$. Suppose $G \cong C_{n_1} \oplus \ldots \oplus C_{n_r}$, where $r \in \N_0$ and $n_1, \ldots, n_r \in \N$ with $1 < n_1 \t \ldots \t n_r$. Let $\mathsf D^*(G)=1+\sum_{i=1}^r(n_i-1)$ and let $(e_1,\ldots,e_r)$ be a basis of $G$ with $\ord(e_i)=n_i$. Then the sequence $X:=e_1^{n_1-1}\cdot\ldots\cdot e_r^{n_r-1}$ is zero-sum free and hence $$\mathsf D(G)\ge |X|+1=1+\sum_{i=1}^r(n_i-1)=\mathsf D^*(G)\,.$$


We have the following well-known result.
\begin{lemma}[{\cite[Theorem 4.2.10]{Ge09a}}]
	\label{Davenport}
	If $\mathsf r(G)\le 2$, then
\begin{equation}\label{eq:d=d*}
\mathsf D(G)=\mathsf D^*(G).
\end{equation}

\end{lemma}

In fact, equality \ref{eq:d=d*} also holds  for $p$-groups and some other special groups. However it is still not known for groups of rank $3$ and for groups of the form $C_n^r$ whether equality \ref{eq:d=d*} holds. On the other hand, there are infinitely many groups $G$ with rank larger than $3$ such that $\mathsf D(G)>\mathsf D^*(G)$ (\cite{C20,B69, GS92, M92}). For more on the Davenport constant, one can see \cite{GG06,Ge09a,Ge-HK06a,[G], Gi18, GS19, GS20,PS11,CS14,S11}.

\subsection{The separating Noether number and separating atoms}\label{subsec:bsep}

Let $H$ be a submonoid of $\mathcal M[V]^G$. If $H$ itself is a separating set of $\C[V]^G$, we say $H$ is a \emph{separating} monoid. Since $\mathcal M[V]^G$ is isomorphic to $\mathcal B(\widehat{G}_{add})$, we consider submonoids $\mathcal M$ of $\mathcal B(\widehat{G}_{add})$ and have the following results.

\begin{lemma}\label{lem:dommain}
Let $G_0$ be a subset of $\widehat{G}_{add}$. Then the following conditions are equivalent for a submonoid $\mathcal{M}$ of $\mathcal{B}(G_0)$:
\begin{itemize}
    \item [(i)] For all $G_1 \subseteq G_0$ we have $\mathsf{q}(\mathcal{M \cap B}(G_1)) \supseteq \mathcal{B}(G_1)$.
    \item [(ii)] For all $G_1 \subseteq G_0$ with $|G_1| \leq \mathsf{r}(G) + 1$ we have $\mathsf{q}(\mathcal{M \cap B}(G_1)) \supseteq \mathcal{B}(G_1)$.
    \item[(iii)] $\mathcal{M}=\psi(S)$ for a subset $S\subseteq\mathcal{M}[V]^G$ that is a separating set for $\mathbb{C}[V]^G$.
 \end{itemize}
\end{lemma}
\begin{proof}
This is a reformulation of \cite[Theorem 2.1]{Do17a}, using also \cite[Corollary 2.3]{domokos-szabo}
\end{proof}

Motivated by Lemma \ref{lem:dommain}, we have the following  definition.

\begin{definition}\label{def:sepmonoid}
    For a subset $G_0\subseteq G$, we call a submonoid $\mathcal{M}$ of $\mathcal{B}(G_0)$ \textit{separating} if for all $G_1 \subseteq G_0$, we have $\mathsf{q}(\mathcal{M \cap B}(G_1)) \supseteq \mathcal{B}(G_1)$.
\end{definition}

Let $\mathcal M$ be a submonoid of $\mathcal B(\widehat{G}_{add})$. Then by Lemma \ref{lem:dommain}, we see that $\mathcal M$ is a separating monoid if and only if there exists a separating submonoid $H$ of $\mathcal{M}[V]^G$ such that $\mathcal M=\psi(H)$. Next we collect some monoid theoretical properties of separating monoids.

\begin{proposition}\label{sepmonoid} For a separating monoid $H\subseteq \mathcal{M}[V]^G$, the following hold.
	\begin{enumerate}
		\item $\mathsf{q}(H)=\mathsf{q}(\mathcal{M}[V]^G)\cong\mathbb{Z}^n$.
		\item  The integral closure $\Tilde{H}$ of $H$ contains $\mathcal{M}[V]^G$.
		\item $H$ is a reduced $C$-monoid (see \cite[Chapters 2 and 3]{Ge-HK06a} for $C$-monoid)
		\item $H$ is Krull if and only if $H=\mathcal{M}[V]^G$.
	\end{enumerate}
\end{proposition}
\begin{proof}
	1. See \cite[Theorem 2.1]{Do17a}.
	
	2.  The normalization of a finitely generated graded separating
subalgebra $\mathbb{C}[V]^G$ (where $G$ is any reductive group) is $\mathbb{C}[V]^G$ (see e.g. \cite[Theorem 2.3.12]{derksen-kemper}). On the other hand,
the normalization of $\mathbb{C}[H]$ is $\mathbb{C}[\Tilde{H}]$ (see \cite[Theorem 4.39]{bruns-gubeladze}).

	3. Since $\mathcal{M}[V]^G$ is reduced by \cite[Proposition 4.7.2]{Cz-Do-Ge16}, we obtain that $H\subseteq \mathcal{M}[V]^G$ is reduced. Since $\mathscr{C}(\widehat{H})=\mathscr{C}(\mathcal{M}[V]^G)$ is finite by \cite[Proposition 4.8.4]{Cz-Do-Ge16}, it follows from \cite[Proposition 4.8]{Ge-Ha08a} that $H$ is a $C$-monoid.
	
	4. The assertion follows from Part 2.
\end{proof}

Note that $\beta(G)=\mathsf D(G)$. Since $\mathcal M[V]^G$ is isomorphic to $\mathcal B(\widehat{G}_{add})$ and hence isomorphic to $\mathcal B(G)$, it is natural to seek a combinatorial characterization of the separating Noether number by  Lemma \ref{lem:dommain}.  Here we follow a reformulated version of the original approach (the original version can be found in \cite[Section 2]{Do17a}, and the reformulated one in \cite[Section 2]{Sc24a}).

Given a monoid $H$ and any subset $H_0\subseteq H$, denote by $[H_0]$ the submonoid generated by $H_0$. Let $G_0\subseteq G$. For any $\ell\in\mathbb{N}$ we define the submonoid
\[\mathcal{B}(G_0)_{\ell}:=[A\in\mathcal{A}(G_0):|A|\leq \ell]\subseteq\mathcal{B}(G_0).\]

\begin{definition}\label{groupatomdef}
For a subset $G_0$ of $G$, we set
\[\mathcal{A}_{\sep}(G_0):=\{A\in\mathcal{A}(G_0):A\notin \mathsf{q}(\mathcal{B}(G_0)_{|A|-1})\}\subseteq \mathcal{A}(G_0).\]
The elements of $\mathcal A_{\sep}(G_0)$ are called \emph{separating atoms} over $G_0$. In particular, we simply say $A$ is a separating atom if $A$ is a separating atom over $\supp(A)$.
\end{definition}
\noindent
Note that the atoms of the monoid $[\mathcal{A}_{\sep}(G_0)]$ are exactly the separating atoms $\mathcal{A}_{\sep}(G_0)$ over $G_0$. Moreover,
\begin{equation}\label{generatingsubset}
	\text{ $\mathsf{q}(\mathcal{B}(G_0))$ is generated as a group by the elements of $\mathcal{A}_{\sep}(G_0)$.  }
\end{equation}
We let
\[
\mathsf D([\mathcal{A}_{\sep}(G_0)])=\max\{|A|\colon A\in \mathcal{A}_{\sep}(G_0)\}\,.
\]
One can express $\sepbeta(G)$ in terms of zero-sum sequences in the following way:

\begin{lemma}[{\cite[Corollary 2.6.]{Do17a}}]\label{md2}
The number ${\sepbeta}(G)$ is the maximal length of an element in $\mathcal{A}_{\sep}(G_0)$, where $G_0$ ranges over all subsets of size $k\leq {\mathsf r}(G)+1$ of the abelian group $G$:
\[{\sepbeta}(G)=\underset{\underset{|G_0|\leq {\mathsf r}(G)+1}{G_0\subseteq G}}{\mathrm{max}}\mathsf{D}([\mathcal{A}_{\sep}(G_0)]).\]
\end{lemma}

The sharpest known lower bound for the separating Noether number of a finite abelian group $G$ was set in \cite[Lemmas 5.2 and 5.5]{Sc25a}.
\begin{lemma}\label{lowerbound}
	Let $G=C_{n_1}\oplus \ldots\oplus C_{n_r}$ with $1<n_1\t \ldots \t n_r$.
	Then
	\[{\sepbeta}(G)\ge \begin{cases}
	n_s+n_{s+1}+\ldots+n_r,&\mbox{ if }r \mbox{ is odd}, \\
	\frac{n_{s}}{p_1}+n_{s+1}+\ldots+n_r,&\mbox{ if }r\mbox{ is even},\\
	\end{cases}\]
	where $s=\lfloor\frac{r+1}{2}\rfloor$ and $p_1$ is the minimal prime divisor of $n_1$.
In particular, we have that ${\sepbeta}(G)>n_{s+1}+\ldots+n_r$.
\end{lemma}
Until now there is no known example, where $\sepbeta(G)$ is strictly larger than the lower bound set in the above Lemma. On the other hand, there are some  families of finite abelian groups, for which $\sepbeta(G)$ is equal to this lower bound. Almost all of these is covered by the following result:
\begin{lemma}[{\cite[Theorem 1.1]{SZZ25a} and \cite[Theroem 1.2]{Sc25a}}]\label{lemain}
	Let $G=C_{n_1}\oplus \ldots\oplus C_{n_r}$ with $1<n_1\t \ldots \t n_r$ and $r\ge 2$, and let $s=\lfloor\frac{r+1}{2}\rfloor$. Suppose either $n_1=n_r$ or $\mathsf D(n_sG)=\mathsf D^*(n_sG)$.
	Then
	\[\begin{cases}
		{\sepbeta}(G)= n_s+n_{s+1}+\ldots+n_r,&\mbox{ if }r \mbox{ is odd} \\
		{\sepbeta}(G)\le \frac{n_{s}}{p}+n_{s+1}+\ldots+n_r,&\mbox{ if }r\mbox{ is even},\\
	\end{cases}\]
	where $p$ is the minimal prime divisor of $n_s$.
\end{lemma}

 Lemma \ref{lemain} and Lemma \ref{lowerbound} imply  that the exact value of $\sepbeta(G)$ is known for the following  finite abelian groups:
\begin{itemize}
    \item the direct sum of $r$ copies of the cyclic group $C_n$ (i.e. if $G=C_n^r$);
    \item finite abelian groups of rank $2,3$ and $5$;
    \item finite abelian $p$-groups.
\end{itemize}
 For more results on ${\sepbeta}(H)$ for general finite groups $H$,
\begin{itemize}
\item see \cite[Theorem 3.10]{Do17a} for the only example of  finite abelian groups for which $\sepbeta(H)$ is known but  not covered by Lemma \ref{lemain};
\item see \cite{domokos-schefler} for results on $\sepbeta(H)$ for finite non-abelian groups;
\item see \cite{Ko-Kr10a} for results on $\sepbeta(H)$ in positive characteristic.
\end{itemize}

\section{Minimal size of monomial separating sets}\label{sec:size}

In this section, we calculate the minimal size of  monomial separating sets of $\mathbb{C}[V]^G$ in some special cases. Results on this topic can be found in \cite{cahillcontreras, neusel-sezer}.

\begin{proposition}\cite[Proposition 3.2]{cahillcontreras}\label{prop:cch}
	Suppose $n=|G|$. There exists a monomial separating set $S$ of $\mathbb{C}[V]^G$ such that for each $L\subseteq[1,n]$ with $|L|\leq{\mathsf r}(G)+1$ there is at most one monomial $\prod_{i=1}^nx_i^{m_i}\in S$ such that $m_i\neq 0$ if and only if $i\in L$ (and no monomials when $|L|>{\mathsf r}(G)+1$).
\end{proposition}

In order to improve Proposition \ref{prop:cch} we characterize those subsets $L\subseteq[1, n]$ for which one monomial is indeed needed in the separating set. Let $G_0\subseteq G$ be a subset with $|G_0|\leq\mathsf{r}(G)+1$. We say $G_0$ has \emph{Property (P)} if
\begin{equation}\tag{P}\label{eq:p}
\mathcal{B}(G_0)\nsubseteq\left\langle\bigcup_{G_1\subsetneq G_0}\mathcal{B}(G_1)\right\rangle\,.
\end{equation}
\begin{proposition}\label{nehez}
For $G_0\subseteq G$ let $\mathcal{M}\subseteq\mathcal{B}(G_0)$ be a separating submonoid. Then
\begin{enumerate}
    \item any generating set of $\mathcal{M}$ contains an element $m_C$ with support $C$ for each $C\subseteq G_0$ where $C$ has Property (P).
    \item  $\mathcal{M}$ does have a generating set of the form $\{m_C\colon C \subseteq G_0$, $C$ has Property (P), where the support of $m_C$ is $C$.
\end{enumerate}
In particular,  the minimal size of a monomial separating set of $\mathbb{C}[V]^G$ is $|\{C\subseteq G \colon C$ has Property (P)$\}|$.
\end{proposition}

\begin{proof}Note that $\widehat{G}_{add}$ is the dual group of $G$. We can view $G_0$ as a subset of $\widehat{G}_{add}$ and apply Lemma \ref{lem:dommain}.

1. The assertion is an immediate consequence of Lemma \ref{lem:dommain} and the definition of the Property (P).

2. Note that by Proposition \ref{prop:cch}, a separating submonoid of $\mathcal{B}(G_0)$ has a generating set with at most one generator $m_C$ for each $C\subseteq G_0$, such that the support of $m_C$ is $C$. By definition of Property (P) (and Lemma \ref{lem:dommain}), the generators $m_C$ where $C$ does not have Property (P) can be omitted.
\end{proof}

Next we establish some conditions, under which Property (P) does not hold.

\begin{lemma}\label{le4.2}
	Let $G_0$ be a subset of $G$.
\begin{enumerate}
    \item If $G_0=\{g_1,g_2\}$ with $g_1\neq g_2$, then it has Property (P) if and only if $\langle g_1\rangle\cap\langle g_2\rangle\neq\{0\}$.
    \item  If $|G_0|\geq 3$ and there exist $g\in G_0$ and $t\in [2, \ord(g)]$ such that $tg\in G_0$, then $G_0$ does not have Property (P).
    \item If $|G_0|\geq \mathsf{r}(\langle G_0\rangle)+2$, then it does not have Property (P).
    \item  If there exists $g\in G_0$ such that $|G_0|\geq \mathsf{r}(\langle G_0\setminus\{g\}\rangle)+2$, then $G_0$ does not have Property (P).
\end{enumerate}
\end{lemma}
\begin{proof}
1. We have $\langle g_1\rangle\cap\langle g_2\rangle\neq\{0\}$  if and only if there exists $A:=g_1^{m_1}g_2^{m_2}\in \mathcal A(\{g_1,g_2\})$ with $m_i\in [1,\ord(g_i)-1]$ for each $i\in \{1,2\}$ (so $A\in \mathcal B(\{g_1,g_2\})$, but $A\not\in \langle g_1^{\ord(g_1)}, g_2^{\ord(g_2)}\rangle$).

2. Let $B\in \mathcal B(G_0)$.
	Then $B=g^{m_1}(tg)^{m_2}W$ for some $m_1,m_2\in \N_0$ and $W\in \mathcal F(G_0\setminus \{g,tg\})$. Let $m_0\in \N$ be minimal such that $m_0\ord(g)>tm_2$. Here $g^{m_1+tm_2}W$, $g^{m_0\ord(g)-tm_2}(tg)^{m_2}$, $g^{m_0\ord(g)}$ are zero-sum sequences, whence
	\begin{center}
	    $B=(g^{m_1+tm_2}W)\cdot (g^{m_0\ord(g)-tm_2}(tg)^{m_2})\cdot (g^{m_0\ord(g)})^{-1}\in \langle \mathcal B(G_0\setminus \{tg\}), \mathcal B(\{g,tg\}) \rangle\,.$
	\end{center}

3. It follows immediatelly from Lemma \ref{lem:dommain}.	

4. By Lemma \ref{le4.2}.3, we may assume that $\mathsf{r}(\langle G_0\rangle)\geq |G_0|-1$. Let $g\in G_0$ such that
	$\mathsf{r}(\langle G_0\setminus\{g\}\rangle)\leq |G_0|-2$ and
	let $H=\langle G_0\setminus\{g\}\rangle$. Since $\mathsf{r}(\langle G_0\rangle)\ge |G_0|-1>\mathsf{r}(H)$, we have $g\notin H$. Let $s$ be the order of the element $g+H\in G/H$. Then $sg\in H$ and for each $M\in\mathcal{B}(G_0)$ we have that $s$ divides $\mathsf v_g(M)$. Let $G'_0=\{sg\}\cup G_0\setminus\{g\}\subseteq H$.
	If $sg\in G_0$, then the assertion follows from Lemma \ref{le4.2}.2. Suppose $sg\not\in G_0$. Then there is a monoid isomorphism between $\mathcal{B}(G_0)$ and $\mathcal{B}(G'_0)$. Since  $|G_0'|=|G_0|\ge \mathsf r(H)+2$, it follows from Lemma \ref{le4.2}.3 that
	\[
	\mathcal{B}(G'_0)\subseteq\left\langle\bigcup_{G'_1\subsetneq G'_0}\mathcal{B}(G'_1)\right\rangle \quad \text{ and hence }\quad  \mathcal{B}(G_0)\subseteq\left\langle\bigcup_{G_1\subsetneq G_0}\mathcal{B}(G_1)\right\rangle\,.
	\]
\end{proof}

In the remaining part of this Section, we apply Proposition \ref{nehez} for some specific abelian groups to obtain precise formulas for the size of a minimal monomial separating set.

Denote by $\phi$ the Euler totient function, and by $\omega(d)$ the number of distinct prime divisors of $d\in\mathbb{N}$.

\begin{proposition}\label{prop:cn}
	Suppose $G$ is a cyclic group of order $n$. Then
the minimal size of  a monomial separating set of $\mathbb{C}[V]^{G}$  is
\[
 n+{n\choose 2}-\sum_{1<d\t n}2^{\omega(d)-1}\phi(d).
\]
\end{proposition}
\begin{proof}
Each subset of $G$ with cardinality $1$ has Property (P). Lemma \ref{le4.2}.1 and \ref{le4.2}.3 implies that it suffices to count the number of possibilities for $G_0=\{g_1,g_2\}$ with $g_1\neq g_2$, such that $\langle g_1\rangle\cap\langle g_2\rangle\neq\{0\}$. Instead of this we subtract from ${n\choose 2}$ the cardinality of
\begin{center}
    $P^*:=\{\{g_1,g_2\}\subseteq G$: $g_1\neq g_2$, $\langle g_1\rangle\cap\langle g_2\rangle=\{0\}\}.$
\end{center}

Observe that $\langle g_1\rangle\cap\langle g_2\rangle=\{0\}$ if and only if $\gcd\left(\ord(g_1),\ord(g_2)\right)=1$. Then
\[
\ord(g_1+g_2)=\lcm (\ord(g_1),\ord(g_2))=\frac{\ord(g_1)\ord(g_2)}{\gcd(\ord(g_1),\ord(g_2))}\,=\ord(g_1)\ord(g_2).
\]
The number of elements of order $d$ is $\phi(d)$; and for every $1<d\in\mathbb{N}$, there are $2^{\omega(d)-1}$ different decompositions $d=d_1d_2$, such that $\mathrm{gcd}(d_1,d_2)=1$ and $d_1<d_2$. It follows that
\begin{align*}
 |P^*|&=\sum_{1<d\t n}\left|\big\{\{g_1,g_2\}\in P^*\colon   d=\ord(g_1+g_2)\big\}\right|\\
 &=\sum_{1<d\t n}\sum_{\ord(h)=d}\left|\big\{\{g_1,g_2\}\in P^* \colon g_1+g_2=h, d=\ord(g_1)\ord(g_2)\big\}\right|\\
&=\sum_{1<d\t n}\phi(d)\left|\big\{\{d_1,d_2\}\subseteq [1,n-1]\colon d_1<d_2, \gcd(d_1,d_2)=1, d=d_1d_2\big\}\right|\\
&=\sum_{1<d\t n}2^{\omega(d)-1}\phi(d)\,.
\end{align*}
\end{proof}

The converse of Lemma \ref{le4.2}.4 is true for finite abelian $p$-groups with an extra condition.

\begin{lemma}\label{kell}Suppose $G$ is a finite abelian $p$-group. Let $G_0$ be a subset of $G$ such that $|G_0|=\mathsf{r}(\langle G_0\rangle)+1$. Then
\begin{center}
    $G_0$ has Property (P) if and only if $\mathsf{r}(\langle G_0\setminus\{g\}\rangle)=\mathsf{r}(\langle G_0\rangle)$ for each $g\in G_0$.
\end{center}
\end{lemma}
\begin{proof}
We start with the sufficiency.
Since $G$ is a $p$-group and $|G_0|=\mathsf{r}(\langle G_0\rangle)+1$, there exists a $g\in G_0$ such that $\langle G_0\rangle=\langle G_0\setminus\{g\}\rangle$. Thus $g\in \langle G_0\setminus\{g\}\rangle$ and hence there exists $A\in \mathcal B(G_0)$ with $\mathsf v_g(A)=1$.

To show that $G_0$ has Property (P), it suffices to prove that for every $h\in G_0\setminus \{g\}$, $p$ divides $\mathsf v_g(M)$ for each $M\in \mathcal B(G_0\setminus \{h\})$.
Let $h\in G_0\setminus\{g\}$.
If $g\in \langle G_0\setminus \{g,h\} \rangle $, then
\[
\mathsf r(\langle G_0\setminus \{h\}\rangle)=\mathsf r(\langle G_0\setminus \{g,h\} \rangle)\le | G_0\setminus \{g,h\} |=\mathsf r(\langle G_0\rangle)-1\,,
\]
a contradiction to our assumption. Thus $g\notin \langle G_0\setminus \{g,h\} \rangle $. Let $H=\langle G_0\setminus \{g,h\} \rangle $, and denote by $t$ the order of the element $g+H\in G/H$. Since $G/H$ is still a $p$-group, $p$ divides $t$ and so for every $M\in \mathcal B(G_0\setminus \{h\})$ we have that $p$ divides $\mathsf v_g(M)$. The assertion follows.

The necessity is a direct consequence of Lemma \ref{le4.2}.4.
\end{proof}

\begin{proposition}\label{prop:eapg}
Let $G$ be an elementary abelian $p$-group of rank $r\geq 2$.
Then the minimal size of a monomial separating set of $\mathbb{C}[V]^{G}$ is
\[
  \lambda(p,r):=p^r+\frac{(p^r-1)(p-2)}{2}+\sum_{i=3}^{r+1}\frac{(p^r-1)(p^r-p)\cdot\ldots\cdot(p^r-p^{i-2})(p-1)^{i-1}}{i!}.
\]

\end{proposition}
\begin{proof}
Each subset of $G$ with cardinality $1$ has Property (P).
Since $G$ is an elementary abelian $p$-group, we have $\langle g_1\rangle\cap\langle g_2\rangle\neq\{0\}$ if only if $\langle g_1\rangle=\langle g_2\rangle\neq\{0\}$. There are $\frac{p^r-1}{p-1}$ cyclic subgroups of $G$, so by Lemma \ref{le4.2}.1, there are $\frac{p^r-1}{p-1}{p-1\choose 2}=\frac{(p^r-1)(p-2)}{2}$ subsets of $G$ with cardinality $2$ having Property (P).

Let $G_0=\{g_1,\ldots,g_{|G_0|}\}\subseteq G$ with $|G_0|\geq 3$ be a subset that has Property (P). Then $|G_0|<\mathsf{r}(\langle G_0\rangle)+2$, since otherwise $G_0$ would not have Property (P) by Lemma \ref{le4.2}.3. If we had $|G_0|=\mathsf{r}(\langle G_0\rangle)$,  then the elements of $G_0$ would form a basis of $\langle G_0\rangle$, since $G$ is an elementary abelian $p$-group. It would follow that $\mathcal{B}(G_0)=[g^p\colon g\in G_0]$ and hence $G_0$ would not have Property (P), a contradiction. Therefore
 $|G_0|=\mathsf{r}(\langle G_0\rangle)+1$ and hence we may assume that $G_0\setminus\{g_1\}$ is a basis of $\langle G_0\rangle$. By Lemma \ref{kell}, $G_0$ has Property (P) if and only if $g_1=\sum_{i=2}^{|G_0|}\alpha_ig_i$  for some $\alpha_2,\ldots,\alpha_{|G_0|}\in [1,p-1]$.
Thus for each $i\in[3,\mathsf{r}(G)+1]$, there are $\frac{(p^r-1)(p^r-p)\cdot\ldots\cdot(p^r-p^{i-2})(p-1)^{i-1}}{i!}$ subsets of $G$ with cardinality $i$ having Property (P). The assertion follows.
\end{proof}

\noindent{\it Remark.}
Comparing our exact value $\lambda(p,r)$ to the upper bound $\mu(p,r)=\sum_{i=1}^{r+1}{p^r\choose i}$ obtained from Proposition \ref{prop:cch}, gives the following.
\begin{enumerate}
    \item For fixed $r$, we have $\underset{p\rightarrow\infty}{\mathrm{lim}}\frac{\lambda(p,r)}{\mu(p,r)}=1$.
    \item For fixed $p$, we have $\underset{r\rightarrow\infty}{\mathrm{lim}}\frac{\lambda(p,r)}{\mu(p,r)}=0$.
\end{enumerate}

\section{Separating Noether number of abelian groups of rank $4$}\label{sec:snrank4}

In this section, we study the exact value of the separating Noether number for finite abelian groups. Our main result is the following proposition.

\begin{proposition}\label{main}
Let $G=C_{n_1}\oplus \ldots\oplus C_{n_r}$ with $1<n_1\t \ldots \t n_r$ and $r\ge 2$. Suppose $\mathsf D(n_sG)=\mathsf D^*(n_sG)$, where $s=\lfloor \frac{r+1}{2}\rfloor$.
\begin{enumerate}
    \item If $G_0$ is a subset of $G$  with $|G_0|\leq r+1$ such that there exists a separating atom $A$ over $G_0$ with $|A|={\sepbeta}(G)$, then $| \supp(A)|=|G_0|=r+1$.
    \item If $r$ is even and $\mathsf D(n_iG)=\mathsf D^*(n_iG)$ for every $i\in [s,r]$, then
    \begin{center}
        ${\sepbeta}(G)=\frac{n_{s}}{p_1}+n_{s+1}+\ldots+n_r$,
    \end{center}
    where $p_1$ is the minimal prime divisor of  $n_1$.
\end{enumerate}
\end{proposition}

The proof of Proposition \ref{main} will follow the ideas of \cite[Theorem 1.1]{SZZ25a}. We need the following lemmas.

\begin{lemma}[{\cite[Lemma 2.2]{SZZ25a}}]\label{separatingatomlength}
Let $G$ be a finite abelian group and let $G_0\subseteq  G$ be a nonempty subset. If $A$ is a separating atom over $G_0$, then $|A|\le \mathsf D^*(G)$.
\end{lemma}
\begin{lemma}[{\cite[Lemma 4.2]{Sc25a}}]\label{ba}
Let $\alpha,\beta,\gamma$ be positive integers with $\gcd(\alpha,\beta) = 1$. Then there exists an $\ell\in \{1,2,\ldots,\alpha\gamma-1\}$, for which $\gcd(\ell,\alpha\gamma) = 1$ and $\ell \beta\equiv 1$ mod $\alpha$ hold.
\end{lemma}

By Lemma \ref{md2} there exists	a subset $G_0\subseteq G$ with $|G_0|\le r+1$ and a separating atom $A$ over $G_0$ with $|A|={\sepbeta}(G)$. Set $G_1=\{n_sg\colon g\in G_0\}$. Define the map
\begin{center}
    $\varphi\colon \{S\in \mathcal F(G_0)\colon n_s\t \mathsf v_g(S) \text{ for each }g\in G_0\}\rightarrow \mathcal F(G_1),$\\
by $\varphi(\prod_{g\in G_0}g^{n_sy_g})=\prod_{g\in G_0}(n_sg)^{y_g},$
\end{center}
where $y_g\in\mathbb{N}$ for each $g\in G_0$. For a sequence $T$ over $G_1$, let $\varphi^{-1}(T)$ denotes the set of all sequences $S$ with $n_s\t \mathsf v_g(S)$ for each $g\in G_0$ such that $\varphi(S)=T$.

Now we are ready to prove Proposition \ref{main}.

\begin{proof}[Proof of Proposition \ref{main}] Let $G_0=\{g_1,\ldots, g_{|G_0|}\}\subseteq G$ be a subset  with $|G_0|\le r+1$ and let
	$$A=\prod_{i=1}^{|G_0|}g_i^{m_i}, \text{ where }m_i\in \N \text{ for each }i\in [1,|G_0|]\,,$$
	be a separating atom  over $G_0$ with $|A|={\sepbeta}(G)$.

We proceed to prove several claims.
\begin{claim}\label{cl1}
If $S\in \mathcal B(G_0)$ with $n_s\t \mathsf v_g(S)$ for each $g\in G_0$, then $S\in \mathsf q(\mathcal B(G_0)_{|A|-1})$.
\end{claim}
\begin{proof}[Proof of \ref{cl1}]
	Since $\varphi(S)$ is a zero-sum sequence over $G_1$, it follows from  \eqref{generatingsubset} that we may factor $\varphi(S)=U_1\cdot\ldots\cdot U_{\ell}\cdot U_{\ell+1}^{-1}\cdot\ldots\cdot U_k^{-1}$, where $U_1,\ldots, U_{k}\in \mathcal A_{\mathrm{sep}}(G_1)$. Therefore by choosing suitable subsequences $\varphi^*(U_i)$ from the set $\varphi^{-1}(U_i)$ for each $i\in [1,k]$, we have
	\[
	S=\varphi^*(U_1)\cdot\ldots\cdot \varphi^*(U_{\ell})\cdot \varphi^*(U_{\ell+1})^{-1}\cdot\ldots\cdot \varphi^*(U_k)^{-1}.
	\]
	In view of  Lemmas \ref{separatingatomlength} and \ref{lowerbound}, for every $i\in [1,k]$, we have
 $$|\varphi^*(U_i)|=n_s|U_i|\le n_s\mathsf D^*(n_sG)= \sum_{j=s+1}^rn_j-(r-s-1)n_s\le \beta_{\mathrm{sep}}(G)-1=|A|-1\,.$$ Now the assertion follows.
	\qedhere(\ref{cl1})	
	\end{proof}

For each $l\in\mathbb{N}_{>0}$ and each $i\in [1, |G_0|]$, there exist $k_i^{(l)}\in \N_0$ and $x_i^{(l)}\in [0, n_s-1]$ such that
\begin{center}
    $lm_i=k_i^{(l)}n_s+x_i^{(l)}$.
\end{center}

\begin{claim}\label{cl2}
There exists some $i_0\in [1, |G_0|]$ such that $x_{i_0}^{(l)}\neq 0$ for any $l\in\mathbb{N}$ with $\gcd(l,n_s)=1$.
\end{claim}	
\begin{proof}[Proof of \ref{cl2}]
If $x_i^{(1)}=0$ for every $i\in [1,|G_0|]$, then $A\in \mathcal B(G_0)$ with $n_s\t \mathsf v_g(S)$ for every $g\in G_0$. Then by \ref{cl1} $A\in \mathsf q(\mathcal B(G_0)_{|A|-1})$, contradicting that $A$ is a separating atom. So there exists some $i_0\in [1, |G_0|]$ with $x_{i_0}^{(1)}\neq 0$. If $\gcd(l,n_s)=1$, then since $x_{i_0}^{(l)}\equiv lx_{i_0}^{(1)}\pmod {n_s}$, we get that $x_{i_0}^{(l)}\neq 0$.
\qedhere[\ref{cl2}]
\end{proof}
	
Set
\begin{center}
    $A^{(l)}:=\prod_{i=1}^{|G_0|}g_i^{k_i^{(l)}n_s}$.
\end{center}
Then $\varphi(A^{(l)})\in\mathcal{F}(n_sG)$, and we can write $\varphi(A^{(l)})=X_0^{(l)}\cdot X_1^{(l)}$, where $X_1^{(l)}\in\mathcal{B}(n_sG)$ and $X_0^{(l)}$ is a zero-sum free sequence over $n_sG$. For suitable $\varphi^*(X_0^{(l)})\in \varphi^{-1}(X_0^{(l)})$ and $\varphi^*(X_1^{(l)})\in \varphi^{-1}(X_1^{(l)})$ we have $A^{(l)}=\varphi^*(X_0^{(l)})\varphi^*(X_1^{(l)})$. It follows that
\begin{equation}\label{X}
|\varphi^*(X_0^{(l)})|=n_s|X_0^{(l)}|\le n_s(\mathsf D(n_sG)-1)=n_s(\mathsf D^*(n_sG)-1)=\sum_{j=s+1}^rn_j-(r-s)n_s\,.
\end{equation}	

Set
\begin{center}
    $W^{(l)}=:\varphi^*(X_0^{(l)})\prod_{i=1}^{|G_0|}g_i^{x_i^{(l)}}$.
\end{center}
\begin{claim}\label{W}
    For any positive integer $l$ with $\gcd(l,n_s)=1$, we have $|W^{(l)}|\ge |A|$.
\end{claim}
\begin{proof}[Proof of \ref{W}]
Note that $W^{(l)}=\varphi^*(X_1^{(l)})^{-1}A^{(l)}\prod_{i=1}^{|G_0|}g_i^{x_i^{(l)}}=\varphi^*(X_1^{(l)})^{-1}A^l\in\mathcal{B}(G_0)$. If $\gcd(l,n_s)=1$, then there exist some $l'\in [1,n_s-1]$ and $h^{(l)}\in \N_0$ such that  $ll'=1+h^{(l)}n_s$, so
\[
A=A^{ll'-h^{(l)}n_s}=(A^{h^{(l)}n_s})^{-1}(A^l)^{l'}=(A^{h^{(l)}n_s})^{-1}(\varphi^*(X_1^{(l)})\cdot W^{(l)})^{l'}\,.
\]
By \ref{cl1}, we have  $A^{h^{(l)}n_s}\in \mathsf q(\mathcal B(G_0)_{|A|-1})$ and $\varphi^*(X_1^{(l)})\in \mathsf q(\mathcal B(G_0)_{|A|-1})$. Since $A$ is a separating atom, we obtain the result.
\qedhere[\ref{W}]
\end{proof}	

Since $m_i\le \ord(g_i)-1$, we have that $k_i^{(l)}\le l\ord(g_i)-1$. Note that
\begin{center}
$\sigma\left(\prod_{i=1}^{|G_0|}g_i^{(l\ord(g_i)-k_i^{(l)})n_s-x_i^{(l)}}\right)=\sigma\left(\prod_{i=1}^{|G_0|}g_i^{l\ord(g_i)n_s}\right)-\sigma(A^l)=0$.
\end{center}
Therefore there exist $t_1,\ldots, t_{|G_0|}\in \N$ with $\sigma\left(\prod_{i=1}^{|G_0|}g_i^{t_in_s-x_i^{(l)}}\right)=0$. Choose a tuple $(t_1^{(l)},\ldots, t_{|G_0|}^{(l)})\in \N^{|G_0|}$ with $\sigma\left(\prod_{i=1}^{|G_0|}g_i^{t_i^{(l)}n_s-x_i^{(l)}}\right)=0$ such that $\sum_{j=1}^{|G_0|}t_j^{(l)}$ is minimal. Set
\begin{center}
    $V^{(l)}=\prod_{i=1}^{|G_0|}g_i^{t_i^{(l)}n_s-x_i^{(l)}} \quad \text{ and }\quad   Y^{(l)}=\prod_{i=1}^{|G_0|}g_i^{(t_i^{(l)}-1)n_s}.$
\end{center}
\begin{claim}\label{V}
$|V^{(l)}|\ge |A|$ for any positive integer $l$ with $\gcd(l,n_s)=1$.
\end{claim}
\begin{proof}[Proof of \ref{V}]
If $\gcd(l,n_s)=1$, then there exist some $l'\in [1,n_s-1]$ and $h^{(l)}\in \N_0$ such that $ll'=1+h^{(l)}n_s$, implying that
\begin{center}
$A=A^{ll'-h^{(l)}n_s}=(A^{h^{(l)}n_s})^{-1}(A^l)^{l'}=(A^{h^{(l)}n_s})^{-1}(V^{(l)})^{-l'}(A^lV^{(l)})^{l'}\,.$
\end{center}
By \ref{cl1}, we have  $A^{h^{(l)}n_s}\in \mathsf q(\mathcal B(G_0)_{|A|-1})$ and $A^lV^{(l)}=\prod_{i=1}^{|G_0|}g_i^{(k_i^{(l)}+t_i^{(l)})n_s}\in \mathsf q(\mathcal B(G_0)_{|A|-1})$. Since $A$ is a separating atom, we obtain the result. \qedhere[\ref{V}]
\end{proof}

\begin{claim}\label{WV}
For any $l\in\mathbb{N}_{>0}$ we have $|W^{(l)}|+|V^{(l)}|\le n_s(|G_0|+2s-2r)+2\sum_{j=s+1}^rn_j.$
\end{claim}
\begin{proof}[Proof of \ref{WV}]
First, we show that $\varphi(Y^{(l)})$ is zero-sum free. Assume for contradiction that $\prod_{i=1}^{|G_0|}(n_sg_i)^{t_i'}$ (with $t_i'\in [0, t_i^{(l)}-1]$ for each $i\in [1,|G_0|]$) is a nontrivial zero-sum subsequence of $\varphi(Y^{(l)})$. So $\sum_{i=1}^{|G_0|}t_i'>0$, $t_i^{(l)}-t_i'\ge 1$, and
\begin{center}
    $\sigma\left(\prod_{i=1}^{|G_0|}g_i^{(t_i^{(l)}-t_i')n_s-x_i^{(l)}}\right)=\sigma(V^{(l)})-\sigma\left(\prod_{i=1}^{|G_0|}(n_sg_i)^{t_i'}\right)=0\,.$
\end{center}
We obtain a contradiction to the minimality of $\sum_{i=1}^{|G_0|}t_i^{(l)}$, so $\varphi(Y^{(l)})$ is zero-sum free. Therefore it follows from our assumption that
\begin{equation}\label{eq2}
n_s(\sum_{i=1}^{|G_0|}(t_i^{(l)}-1))=|Y^{(l)}|=n_s|\varphi(Y^{(l)})|\le n_s(\mathsf D(n_sG)-1)=\sum_{j=s+1}^{r}n_j-(r-s)n_s\,.
\end{equation}
Combining \eqref{X} and \eqref{eq2} implies that for any positive $l$
\begin{center}
$|W^{(l)}|+|V^{(l)}|=\left(|\varphi^*(X_0^{(l)})|+\sum_{i=1}^{|G_0|}x_i^{(l)}\right)+\left(n_s\sum_{i=1}^{|G_0|}t_i^{(l)}-\sum_{i=1}^{|G_0|}x_i^{(l)}\right)\le$\\
$\le \sum_{j=s+1}^rn_j-(r-s)n_s+\sum_{j=s+1}^rn_j-(r-s)n_s+|G_0|n_s=n_s(|G_0|+2s-2r)+2\sum_{j=s+1}^rn_j.$
\end{center}
\qedhere[\ref{WV}]
\end{proof}
Suppose that the minimal prime divisor $p_1$ of $n_1$ is strictly greater than the minimal prime divisor $p$ of $n_s$. For \ref{clm:m} and \ref{cl:n1} we make the following assumptions:
\begin{equation}\label{asptn}
    {\mathsf r}(G)={\mathsf r}(G)=r\text{ is even}\quad\text{and}\quad\frac{n_{s}}{p_1}+n_{s+1}+\ldots+n_r<|A|\leq  \frac{n_{s}}{p}+n_{s+1}+\ldots+n_r.
\end{equation}

\begin{claim}\label{clm:m}
If \eqref{asptn} holds, then there exists $m\in\mathbb{N}$ with $\gcd(n_1,m)=1$ such that $A^m\in \mathsf q(\mathcal B(G_0)_{|A|-1}))$.
\end{claim}
\begin{proof}[Proof of \ref{clm:m}]
Let $d=\gcd(|A|,n_{s})$. Then $|A|=n_{s+1}+\ldots+n_r+bd$ for some $b\in \N$ with $\gcd(b,\frac{n_{s}}{d})=1$. Applying Lemma \ref{ba} with $\alpha =\frac{n_{s}}{d}$, $\beta =b$, and $\gamma=d$, there exists $l\in [1, n_s-1]$ such that $\gcd(l,n_s) = 1$ and $\ell b \equiv 1 \mod \frac{n_{s}}{d}$, whence
\begin{align}\label{ns}
    \ell bd\equiv d \mod n_s\ \text{ and }\ |W^{(l)}|\equiv l|A|\equiv lbd\equiv d  \mod n_s.
\end{align}
Since $r$ is even and $|G_0|\leq r+1$, it follows from \ref{W}, \ref{V} and \ref{WV} that
\begin{center}
    $\frac{n_s}{p_1}+n_{s+1}+\ldots+n_r< |A|\le |W^{(l)}|= (|W^{(l)}| + |V^{(l)}|)-|V^{(l)}|\leq$\\ $\leq2(n_{s+1}+\ldots+n_r)+n_s-|A|< n_{s+1}+\ldots+n_r+n_s-\frac{n_s}{p_1}$.
\end{center}
Thus \eqref{ns} implies  that
\begin{center}
    $|W^{(l)}|=n_{s+1}+\ldots+n_r+d\geq |A|=n_{s+1}+\ldots+n_r+bd.$
\end{center}
It follows that $b=1$ and
\begin{center}
    $\frac{n_{s}}{p_1}+n_{s+1}+\ldots+n_r<n_{s+1}+\ldots+n_r+d=|A|\leq  \frac{n_{s}}{p}+n_{s+1}+\ldots+n_r,$ so
\end{center}
\begin{equation}\label{eq:d}
    \frac{n_{s}}{p_1}<d\leq  \frac{n_{s}}{p}.
\end{equation}
Introduce the notation $m:=\frac{n_s}{d}$. Then $m<p_1$ by \eqref{eq:d}, so $\gcd(n_1,m)=1$.

By (\ref{WV}) we have
\begin{center}
    $\min\{|W^{(m)}|,|V^{(m)}|\}\leq \frac{|W^{(m)}V^{(m)}|}{2}\leq n_{s+1}+\ldots+n_r+\frac{n_s}{2}\,.$
\end{center}
Note that
\begin{center}
$|W^{(m)}|\equiv |A^{m}|=m|A|=m(n_{s+1}+\ldots+n_r+\frac{n_s}{m})\equiv 0  \ \mod n_s\,.$
\end{center}
Since $|W^{(m)}V^{(m)}|\equiv 0  \ \mod n_s$, we have $|V^{(m)}|\equiv 0  \ \mod n_s$. Therefore
\begin{center}
$\min\{|W^{(m)}|,|V^{(m)}|\}\leq n_{s+1}+\ldots+n_r<\frac{n_s}{p_1}+n_{s+1}+\ldots+n_r< |A|,$
\end{center}
so $W^{(m)}\in \mathsf q(\mathcal B(G_0)_{|A|-1}) \text{ or}\ V^{(m)}\in \mathsf q(\mathcal B(G_0)_{|A|-1}).$ Moreover, $\varphi^*(X_1^{(m)})\in \mathsf q(\mathcal B(G_0)_{|A|-1})$ and $(W^{(m)}V^{(m)})\in \mathsf q(\mathcal B(G_0)_{|A|-1})$ by \ref{cl1}. So the equality
\begin{center}
    $A^{m}=W^{(m)}\varphi^*(X_1^{(m)})
	=(W^{(m)}V^{(m)})(V^{(m)})^{-1}\varphi^*(X_1^{(m)})$
\end{center}
shows that $A^{m}\in \mathsf q(\mathcal B(G_0)_{|A|-1}).$
\qedhere[\ref{clm:m}]
\end{proof}

\begin{claim}\label{cl:n1}
If \eqref{asptn} holds, then $A^{n_1}\in \mathsf q(\mathcal B(G_0)_{|A|-1})$.
\end{claim}
\begin{proof}[Proof of \ref{cl:n1}]
Since $B:=\prod_{i=1}^{|G_0|}(n_1g_i)^{m_i}\in \mathcal B(n_1G_0)$, by \eqref{generatingsubset} we have $B\in \mathsf q(\mathcal B( n_1G_0)_{{\sepbeta}(n_1G)})$. Suppose that $B=B_1\cdot\ldots\cdot B_{\ell}\cdot B_{\ell+1}^{-1}\cdot\ldots\cdot B_t^{-1}$, where each $|B_i|\leq \beta_{\mathrm {sep}}(n_1G)$. Thus there exist $A_1,\ldots, A_t$ with each $|A_i|=n_1|B_i|\leq n_1{\sepbeta}(n_1G)$ such that
\begin{center}
$A^{n_1}=A_1\cdot\ldots\cdot A_{\ell}\cdot A_{\ell+1}^{-1}\cdot\ldots\cdot A_t^{-1}\in \mathsf q(\mathcal B(G_0)_{n_1\sepbeta(n_1G)})$
\end{center}
Note that $n_1G\cong C_{n_2/n_1}\oplus \ldots\oplus C_{n_r/n_1}$. We have
\begin{center}
    $\mathsf D(\frac{n_j}{n_1}n_1G)=\mathsf D(n_jG)=\mathsf D^*(n_jG)=\mathsf D^*(\frac{n_j}{n_1}n_1G) \text{ for every }j\in [s,r]\,.$
\end{center}
In particular, it holds for $s'=\left\lfloor \frac{{\mathsf r}(n_1G)+1}{2} \right\rfloor+r-{\mathsf r}(n_1G)\ge s+1$. By Lemma \ref{lemain},
\begin{center}
    $n_1{\sepbeta}(n_1G)=n_{s'}+n_{s'+1}+\ldots+n_{r}\le  n_{s+1}+n_{s+2}+\ldots+n_{r}\le  {\sepbeta}(G)-1=|A|-1,$
\end{center}
which implies that $A^{n_1}\in \mathsf q(\mathcal B(G_0)_{|A|-1})$.
\qedhere[\ref{cl:n1}]
\end{proof}

Now we can prove our main assertions.

1.  Combining \ref{W}, \ref{V}, \ref{WV} and Lemma \ref{lowerbound} yields that
\[n_s\left(\frac{|G_0|}{2}+\left\lfloor\frac{r+1}{2}\right\rfloor-r\right)+\sum_{j=s+1}^rn_j\geq |A|={\sepbeta}(G)\\
\ge \begin{cases}
n_s+n_{s+1}+\ldots+n_r,&\mbox{ if }r \mbox{ is odd}, \\
\frac{n_{s}}{p_1}+n_{s+1}+\ldots+n_r,&\mbox{ if }r\mbox{ is even},\\
\end{cases}\]
where $p_1$ is the minimal prime divisor of $n_1$. Assume to the contrary that $|G_0|\leq r$. If  $r$ is odd, then $n_s\leq n_s(\frac{r}{2}+\frac{r+1}{2}-r)=\frac{n_s}{2}$, a contradiction. If  $r$ is even, then $\frac{n_{s}}{p_1}\leq n_s(\frac{r}{2}+\frac{r}{2}-r)=0$, a contradiction. Therefore $|G_0|=r+1$.

2. Suppose that $r$ is even. By Lemma \ref{lemain} and Lemma \ref{lowerbound} we have that
\begin{center}
    $\frac{n_{s}}{p_1}+n_{s+1}+\ldots+n_r\leq {\sepbeta}(G)=|A|\leq  \frac{n_{s}}{p}+n_{s+1}+\ldots+n_r,$
\end{center}
If $p=p_1$, then we are done. Assume that $p<p_1$ and assume for contradiction that
\begin{center}
    $\frac{n_{s}}{p_1}+n_{s+1}+\ldots+n_r<|A|\leq  \frac{n_{s}}{p}+n_{s+1}+\ldots+n_r$.
\end{center}
Then \eqref{asptn} is satisfied. Since $\gcd(m,n_1)=1$, there exist $\lambda_1,\lambda_2\in\mathbb{Z}$ such that $\lambda_1n_1+\lambda_2m=1$. Therefore $A=(A^{n_1})^{\lambda_1}(A^m)^{\lambda_2}$, so by \ref{clm:m} and \ref{cl:n1} $A\in\mathsf{q}(\mathcal{B}(G_0)_{|A|-1})$, hence $A$ is not a separating atom over $G_0$. The contradiction shows that ${\sepbeta}(G)=\frac{n_{s}}{p_1}+n_{s+1}+\ldots+n_r$.
\end{proof}
\begin{proof}[Proof of Theorem \ref{thm:rank4}]
	Since $\mathsf{r}(G)=4$, we obtain $\mathsf{r}(n_iG)\le 2$ for $i\in [2,4]$, whence $\mathsf D(n_iG)=\mathsf D^*(n_iG)$ by  Lemma \ref{Davenport}.
	The assertion now follows  from  Proposition \ref{main}.2.
\end{proof}
Finally, we mention the following conjecture.
\begin{conjecture}\label{r+1}
Let $G=C_{n_1}\oplus \ldots\oplus C_{n_r}$ with $1<n_1\t \ldots \t n_r$. Let $A$ be a separating atom over $G_0$ with $|A|={\sepbeta}(G)$, where $G_0\subseteq G$ is a subset with $|G_0|\le r+1$. Then $|\supp(A)|=|G_0|=r+1.$
\end{conjecture}

If the above conjecture holds, then for any $M\in \mathcal B(G_0)$ with $|\supp(M)|\leq r$, we have
\begin{align}\label{MB}
M\in \mathsf q(\mathcal B(G_0)_{|A|-1}).
\end{align}

\section{Inverse problem of $\sepbeta(G)$ for abelian groups of rank $2$} \label{sec:invprob}
In this section, we consider the inverse problem concerning ${\sepbeta}(G)$, namely we investigate the structure of separating atoms with maximal length. 
In \cite{Sc24a,Sc25a}, the first author studied the inverse problem and got the following result.

\begin{lemma}[{\cite[Proposition 4.4]{Sc24a} and \cite[Theorem 6.2]{Sc25a}}] \label{invsh}
	Let $G=C_{n_1}\oplus \ldots\oplus C_{n_r}$ with $1<n_1\t \ldots \t n_r$ and let $A$ be a separating atom with $|\supp(A)|\le r+1$ and $|A|=\sepbeta(G)$.
	\begin{enumerate}
		\item If $n_s=\ldots=n_r$, where $s=\lfloor \frac{r+1}{2}\rfloor$, then $\ord(g)=n_r$ for every $g\in \supp(A)$.
		
		\item If $r=2$, then $|\supp(A)|=3$ and  either $\ord(g_1)=\ord(g_2)=\ord(g_3)=n_2$ or $\ord(g_1)=\ord(g_2)=n_2$, $\ord(g_3)=n_1$, where $\supp(A)=\{g_1,g_2,g_3\}$ with $\ord(g_1)\ge \ord(g_2)\ge \ord(g_3)$.
	\end{enumerate}
\end{lemma}

To prove Theorem \ref{max2}, we need the  following proposition.
\begin{proposition}\label{max}
	Let $G=C_{n_1}\oplus \ldots\oplus C_{n_r}$ with $1<n_1\t \ldots \t n_r$ and let $A$ be a separating atom with $|\supp(A)|\le r+1$ and $|A|=\sepbeta(G)$.
	\begin{enumerate}
		\item If $n_s=\ldots =n_{r-1}<n_r$ and $\ord(g)=n_r$ for every $g\in \supp(A)$, where $s=\lfloor \frac{r+1}{2}\rfloor$, then $\frac{n_r}{n_{r-1}}|(r-1)$.
		\item If Conjecture \ref{r+1} holds for $G$, then $g\not \in \langle g'\rangle$ for any two distinct elements $g,g'\in \supp(A)$.
	\end{enumerate}
\end{proposition}

\begin{proof} Let $G_0=\{g_1,\ldots,g_{|G_0|}\}\subseteq G$ be a subset with $|G_0|\le r+1$ and let
 $$A=\prod_{i=1}^{|G_0|}g_i^{m_i}, \text{ where }m_i\in \N \text{ for each }i\in [1, |G_0|]\,,$$ be a separating atom with $|A|=\sepbeta(G)$.

Suppose that $n_s=\ldots=n_{r-1}<n_r$ and that $\ord(g)=n_r$ for every $g\in G_0$. Then $G = H\oplus \langle g^*\rangle \cong H\oplus C_{n_r}$ for some subgroup $H\subseteq G$ with $\exp(H)=n_{r-1}<n_r$ and some $g^*\in G$ with $\ord(g^*)=n_r$, which implies that $\langle n_{r-1}g_i\rangle \subseteq \langle g^*\rangle$ is a subgroup of order $\frac{n_r}{n_{r-1}}$ for each $i\in [1, |G_0|]$. It follows that $\langle n_{r-1}g_i\rangle = \langle n_{r-1}g^*\rangle$ for each $i\in [1,|G_0|]$.
	
Let $H=\langle g_{1}\rangle\cap\ldots\cap \langle g_{|G_0|}\rangle$. Then $H$ is a cyclic group with $\langle n_{r-1}g^*\rangle\subseteq H$, therefore
\begin{equation}\label{m}
\frac{n_r}{n_{r-1}} \quad \text{ divides }\quad  |H|\,.
\end{equation}
Let $m=|H|$, $m^*=\frac{n_r}{m}$, and let $h_j=m^* g_{j}$ for every $j\in [1,|G_0|]$. Then $\langle h_1\rangle=\ldots=\langle h_{|G_0|}\rangle=H$,  $G_1:=\{m^*g\colon g\in G_0\}\subseteq H$, and
\begin{equation}\label{m^*}
m^*\quad  \text{ divides }\quad  n_{r-1}\,.
\end{equation}
Define the map
\begin{center}
$\varphi\colon \{S\in \mathcal F(G_0)\colon m^*\t \mathsf v_g(S) \text{ for each }g\in G_0\}\rightarrow \mathcal F(G_1)$,\\
by $\varphi(\prod_{g\in G_0}g^{m^*y_g})=\prod_{g\in G_0}(m^*g)^{y_g}$,
\end{center}
where $y_g\in\mathbb{N}$ for each $g\in G_0$. For a sequence $T$ over $G_1$, let $\varphi^{-1}(T)$ denote the set of all sequences $S$ with $n_s\t \mathsf v_g(S)$ for each $g\in G_0$ such that $\varphi(S)=T$.
	
\begin{clam}\label{cl3}
Let $S\in \mathcal B(G_0)$ with $m^*\t \mathsf v_g(S)$ for each $g\in G_0$. Then $S\in \mathsf q(\mathcal B(G_0)_{|A|-1})$.
\end{clam}
\begin{proof}[Proof of \ref{cl3}]
This is similar to the proof of \ref{cl1}. The only difference is that now we have to use the inequality
$m^*\mathsf D^*(m^*G)=m^*\mathsf{D}(H)=n_r\leq\sepbeta(G)-1=|A|-1$.
\qedhere[\ref{cl3}]
\end{proof}
There exist $u_i\in \N_0$ and $x_i\in [0, m^*-1]$ such that $m_i=u_im^*+x_i$. Similarly to \ref{cl2}, there exists some $i_0\in [1, |G_0|]$ such that  $x_{i_0}\neq 0$. Since $A$ is a separating atom over $G_0$,
\begin{center}
$A_1:=\prod_{i=1}^{|G_0|}(m^*g_i)^{u_i}=\prod_{i=1}^{|G_0|}h_i^{u_i}$
\end{center}
is zero-sum free over $G_1$, so
\begin{align}\label{u}
|A_1|=\sum_{i=1}^{|G_0|}u_i\leq \mathsf {D}(H)-1=m-1.
\end{align}
Note that $m_i\le \ord(g_i)-1$. We have that
\begin{center}
    $\sigma\left(\prod_{i=1}^{|G_0|}g_i^{(m-u_i)m^*-x_i}\right)
	=\sigma\left(\prod_{i=1}^{|G_0|}g_i^{\ord(g_i)}\right)-\sigma(A)=0\,,$
\end{center}
so there exist $t_i\in [1,m-u_i]$ for each $i\in [1, |G_0|]$ such that $\sigma\left(\prod_{i=1}^{|G_0|}g_i^{t_im^*-x_i}\right)=0$. Choose a tuple $(v_1,\ldots, v_{|G_0|})\in [1,m-u_1]\times \ldots \times [1,m-u_{|G_0|}]$ with $\sigma\left(\prod_{i=1}^{|G_0|}g_i^{v_im^*-x_i}\right)=0$ such that $\sum_{j=1}^{|G_0|}v_j$ is minimal. Set
\begin{center}
    $V=\prod_{i=1}^{|G_0|}g_i^{v_im^*-x_i}\in\mathcal{B}(G_0) \quad \text{ and }\quad   Y=\prod_{i=1}^{|G_0|}g_i^{(v_i-1)m^*}\in\mathcal{F}(G_0)$.
\end{center}
\begin{clam}\label{lem:zsf}
$\varphi(Y)$ is zero-sum free over $H$.
\end{clam}
\begin{proof}[Proof of \ref{lem:zsf}]
Assume to the contrary that $\prod_{i=1}^{|G_0|}(m^*g_i)^{v_i'}$ with $v_i'\in [0, v_i-1]$ for every $i\in [1,|G_0|]$ is a nontrivial zero-sum subsequence $\varphi(Y)$. Therefore $\sum_{i=1}^{|G_0|}v_i'>0$, $v_i-v_i'\in [1,m-u_i]$ for every $i\in [1,|G_0|]$, and hence
\begin{center}
    $\sigma\left(\prod_{i=1}^{|G_0|}g_i^{(v_i-v_i')m^*-x_i}\right)
	=\sigma(V)-\sigma\left(\prod_{i=1}^{|G_0|}(m^*g_i)^{v_i'}\right)=0\,,$
\end{center}
a contradiction to the minimality of $\sum_{i=1}^{|G_0|}v_i$. So $\varphi(Y)$ is zero-sum free over $H$. \qedhere[\ref{lem:zsf}]
\end{proof}

It follows from \ref{lem:zsf} that
\begin{align}\label{v}
|\varphi(Y)|=\sum_{i=1}^{|G_0|}(v_i-1)\leq \mathsf {D}(H)-1=m-1.
\end{align}
Since $v_i\in [1,m-u_i]$ for every $i\in [1,|G_0|]$, we have that
\begin{align}\label{AVAt}\notag
\prod_{i=1}^{|G_0|}g_i^{\ord(g_i)}=AVA',\text{ where}\\
A'=\prod_{i=1}^{|G_0|}g_i^{m^*s_i}\in \mathcal B(G_0)
\end{align}
with each $s_i=m -u_i-v_i\in [0,m-1]$. By \ref{cl3}, we have  $AV=\prod_{i=1}^{|G_0|}g_i^{(u_i+v_i)m^*}\in \mathsf q(\mathcal B(G_0)_{|A|-1})$. Since $A$ is a separating atom, it follows that $|V|\ge |A|$, and then
\begin{align}\label{2A}
2|A|\le |A|+|V|=\sum_{i=1}^{|G_0|}(u_i+v_i)m^*.
\end{align}
\begin{clam}\label{cl4}
We have $\sum_{i=1}^{|G_0|}u_i=\sum_{i=1}^{|G_0|}(v_i-1)=m-1$.
\end{clam}

\begin{proof}[Proof of \ref{cl4}]
If $\sum_{i=1}^{|G_0|}u_i\leq m-2$ or $\sum_{i=1}^{|G_0|}(v_i-1)\leq m-2$, then since $n_s=\ldots =n_{r-1}<n_r$ and $\ord(g)=n_r$ for every $g\in G_0$, \eqref{u}, \eqref{2A}, \eqref{v}, \eqref{m^*}, and  $|G_0|\leq r+1$ yields that
\begin{center}
    $2|A|\le \sum_{i=1}^{|G_0|}(u_i+v_i)m^*\leq m^*(2m-3+|G_0|)\leq 2n_r+m^*(r-2)\leq 2n_r+(r-2)n_{r-1}.$
\end{center}
Letting $p_1$ be the minimal prime divisor of $n_1$, it follows from Lemma \ref{lowerbound} that
\begin{center}
$n_r+\frac{r-2}{2}n_{r-1}\geq |A|={\sepbeta}(G)\ge
\begin{cases}
n_r+\frac{r-1}{2}n_{r-1},&\mbox{ if }r \mbox{ is odd}, \\
n_r+\frac{r-2}{2}n_{r-1}+\frac{n_{r-1}}{p_1},&\mbox{ if }r\mbox{ is even},
\end{cases}$
\end{center}		
a contradiction. Thus $\sum_{i=1}^{|G_0|}u_i\ge m-1$ and  $\sum_{i=1}^{|G_0|}(v_i-1)\ge m-1$, so we are done by \eqref{u} and \eqref{v}. \qedhere[\ref{cl4}]
\end{proof}	

\begin{clam}\label{cl:m}
We have $\frac{n_r}{n_{r-1}}=m.$
\end{clam}
\begin{proof}[Proof of \ref{cl:m}]
By \eqref{m}, we have that $\frac{n_r}{n_{r-1}}\t m$. Assume for contradiction that  $\frac{n_r}{n_{r-1}}<m$.  Then $\frac{n_r}{n_{r-1}}\leq \frac{m}{2}$ and so $m^*\le \frac{n_{r-1}}{2}$. Combining (\ref{2A}), (\ref{u}), (\ref{v}), and $|G_0|\leq r+1$ yields that
\begin{center}
$2|A|\le \sum_{i=1}^{|G_0|}(u_i+v_i)m^*\leq m^*(2m-1+r)\leq 2n_r+m^*(r-1)\leq 2n_r+\frac{r-1}{2}n_{r-1}.$
\end{center}		
Letting $p_1$ be the minimal prime divisor of $n_1$, it follows from Lemma \ref{lowerbound} that
\begin{center}
$n_r+\frac{r-1}{4}n_{r-1}\geq |A|={\sepbeta}(G)\ge
\begin{cases}
n_r+\frac{r-1}{2}n_{r-1},&\mbox{ if }r \mbox{ is odd}, \\
n_r+\frac{r-2}{2}n_{r-1}+\frac{n_{r-1}}{p_1},&\mbox{ if }r\mbox{ is even},
\end{cases}$
\end{center}
a contradiction. So $\frac{n_r}{n_{r-1}}=m$. \qedhere[\ref{cl:m}]
\end{proof}

\begin{clam}\label{cl:hzs}
$\varphi(A')=h^{(|G_0|-2)(m-1)}$ for some $h\in H$ with $\ord(h)=m$.
\end{clam}
\begin{proof}[Proof of \ref{cl:hzs}]
By \ref{cl4} and  (\ref{AVAt}), we have
\begin{align}\label{tlow}
|\varphi(A')|=\sum_{i=1}^{|G_0|}s_i=\sum_{i=1}^{|G_0|}(m -u_i-v_i)=(|G_0|-2)(m-1).
\end{align}
By \ref{cl4} and \eqref{v}, $\varphi(Y)$ is zero-sum free over $H$ with $|\varphi(Y)|=m-1=\mathsf D(H)-1$. It follows that there exists $h\in H$ with $\ord(h)=m$ such that
\begin{center}
$\varphi(Y) =h^{m-1}.$
\end{center}
We show that $m^*g_{i}\neq h$ for each $s_i\ge 1$. Assume for contradiction that there exists $i_0\in [1,|G_0|]$ with $s_{i_0}\ge 1$ such that $m^*g_{i_0}=xh$ for some $x\in [2, m-1]$. Observe that
\begin{center}
$\varphi(YA')=\varphi\left(\prod_{i=1}^{|G_0|}g_i^{(v_i-1)m^*}\boldsymbol{\cdot}\prod_{i=1}^{|G_0|}g_i^{m^*s_i}\right) =h^{m-1}\prod_{i=1}^{|G_0|}(m^*g_i)^{s_i}\in \mathcal F(H)$,	
\end{center}
so $T=(xh)h^{m-x-1}$ is a
subsequence of $\varphi(YA')$ with sum $\sigma((xh)h^{m-x-1})=-h=\sigma(Y)$. Therefore there exists a subsequence $S=\prod_{i=1}^{|G_0|} g_i^{m^*v_i'}$ of $YA'=\prod_{i=1}^{|G_0|} g_i^{m^*(m-u_i-1)}$ with $\varphi(S)=T$. Then we have:
\begin{itemize}
    \item $v_1',\ldots, v_{|G_0|}'\in [0, m-u_i-1]$, so $(v_1'+1,\ldots, v_{|G_0|}'+1)\in [1,m-u_1]\times \ldots \times [1,m- u_{|G_0|}]$,
    \item $\sum_{i=1}^{|G_0|}(v_i'+1)=m-x+|G_0|<m-1+|G_0|=\sum_{i=1}^{|G_0|}v_i$ (by \ref{cl4}),
    \item $\sigma\left(\prod_{i=1}^{|G_0|} g_i^{m^*(v_i'+1)-x_i}\right)=\sigma\left(\prod_{i=1}^{|G_0|} g_i^{m^*v_i'}\cdot\prod_{i=1}^{|G_0|} g_i^{(m^*v_i-x_i)-m^*(v_i-1)}\right)=\sigma(SVY^{-1})=\sigma(\varphi(S))+0-\sigma(Y)=\sigma(T)-\sigma(Y)=0$.
\end{itemize}
So $\prod_{i=1}^{|G_0|} g_i^{m^*(v_i'+1)-x_i}$ is zero-sum sequence contradicting the minimality of $\sum_{i=1}^{|G_0|}v_i$. Therefore  $m^*g_{i}=h$ for each $s_i\ge 1$, and since $\sum_{i=1}^{|G_0|}s_i=(|G_0|-2)(m-1)$, we are done.
\qedhere[\ref{cl:hzs}]
\end{proof}

Now we are ready to show the main assertions.

1. Suppose that $n_s=\ldots=n_{r-1}<n_r$ and that $\ord(g)=n_r$ for every $g\in G_0$. It follows from \ref{cl:hzs} the existence of a zero-sum sequence $\varphi(A')=h^{(|G_0|-2)(m-1)}$ with $\ord(h)=m$. Moreover, $m=\frac{n_r}{n_{r-1}}$ by \ref{cl:m} and $|G_0|=r+1$ by Proposition \ref{main}.1. So $\frac{n_r}{n_{r-1}}$ divides $r-1$.

2. Assume to the contrary that there exist distinct $i, j\in [1,|G_0|]$ such that $g_i\in \langle g_j\rangle$.
	Suppose $g_i=xg_j$, $A=g_i^{m_i}g_j^{m_j}B=(xg_j)^{m_i}g_j^{m_j}B$, and $xm_i\equiv x_i \mod \ord(g_j)$, for some  $x, x_i\in [0,\ord(g_j)-1]$, and $B\in \mathcal F(G_0\setminus \{g_i,g_j\})$. Thus $$A=((xg_j)^{m_i}g_j^{\ord(g_j)-x_i})(Bg_j^{m_j+x_i})(g_j^{\ord(g_j)})^{-1}.$$
	Since $(xg_j)^{m_i}g_j^{\ord(g_j)-x_i}$ is a product of minimal zero-sum subsequences over $\langle g_j\rangle$ and each minimal zero-sum subsequence has length at most  $\ord(g_j)<|A|$, we have that $$(xg_j)^{m_i}g_j^{\ord(g_j)-x_i}\in \mathsf q(\mathcal B(G_0)_{|A|-1}).$$
	Note that $|\supp(Bg_j^{m_j+x_i})|=|G_0|-1\leq r$, so by Conjecture \ref{r+1} and (\ref{MB}) we have $$Bg_j^{m_j+x_i}\in \mathsf q(\mathcal B(G_0)_{|A|-1}).$$
	Since $\ord(g_j)<|A|$, we have $g_j^{\ord(g_j)}\in \mathsf q(\mathcal B(G_0)_{|A|-1}).$
	Therefore $$A\in \mathsf q(\mathcal B(G_0)_{|A|-1}),$$ contradicting that $A$ is a separating atom.
	This completes the proof.
\end{proof}

\begin{proof}[Proof of Theorem \ref{max2}]
Since $\mathsf{r}(G)=2$, we have $s:=\left\lfloor\frac{\mathsf{r}(G)+1}{2}\right\rfloor=1$, which implies that $n_sG$ is cyclic and hence $\mathsf D(n_sG)=\mathsf D^*(n_sG)$ by Lemma \ref{Davenport}. It follows from Proposition \ref{main}.1 that $|\supp(A)|=3$, say $\supp(A)=\{g_1,g_2,g_3\}$ with $\ord(g_1)\ge \ord(g_2)\ge \ord(g_3)$, whence Conjecture \ref{r+1} holds for $G$. By Proposition \ref{max}.2, we have $g_i\not \in \langle g_j\rangle$ for any two distinct indexes $i,j\in [1,3]$.
	
	 It remains to show $\ord(g_1)=\ord(g_2)=n_2$ and $\ord(g_3)=n_1$.
	If $n_1=n_2$, then the assertion follows from Lemma \ref{invsh}.1.
	Suppose $n_1<n_2$. Assume to the contrary that the assertion fails. Then Lemma \ref{invsh}.2 implies that $\ord(g_1)=\ord(g_2)=\ord(g_3)=n_2$. It follows from  Proposition \ref{max}.1  that $1<\frac{n_r}{n_{r-1}}|(r-1)=1$, a contradiction.
\end{proof}

\section*{Acknowledgement}
The authors thank Alfred Geroldinger for sharing his thoughts on the topic of the manuscript and the referee for very valuable recommendations.

\providecommand{\bysame}{\leavevmode\hbox to3em{\hrulefill}\thinspace}
\providecommand{\MR}{\relax\ifhmode\unskip\space\fi MR }
\providecommand{\MRhref}[2]{%
	\href{http://www.ams.org/mathscinet-getitem?mr=#1}{#2}
}

\end{document}